\documentclass[11pt]{amsart}
\usepackage{amsmath,amssymb,amsfonts,url,mathptmx}
\usepackage[mathscr]{euscript}
 \usepackage{hyperref}
 \usepackage{mathrsfs}
\usepackage{graphicx}
\usepackage{color}
  \usepackage{epsfig}
\numberwithin{equation}{section}

\newtheorem{thm}{Theorem}[section]
\newtheorem{lem}{Lemma}[section]

\newtheorem{prop}{Proposition}[section]
\newtheorem{cor}{Corollary}[section]

\begin{document}
\title[Elliptic System Flow]{Long Time Existence of A Flow of Elliptic Systems } \subjclass{35J60, 35J47}
\keywords{Liouville system, Entropy formula, Toda system. Flow, asymptotic analysis, a priori estimate, classification of solutions, blowup phenomenon, Long time existence, Parabolic system, elliptic systems}

\author{Woongbae Park} 
\address{Department of Mathematics\\
University of Rochester\\
803 Hylan Building\\
Rochester, NY 14627}
\email{wpark14@ur.rochester.edu}

\author{Lei Zhang}\footnote{Lei Zhang is partially supported by Simons Foundation Grant SFI-MPS-TSM-00013752}
\address{Department of Mathematics\\
        University of Florida\\
        358 Little Hall P.O.Box 118105\\
        Gainesville FL 32611-8105}
\email{leizhang@ufl.edu}

\date{\today}

\begin{abstract}  For elliptic systems defined on Riemann surfaces, Liouville and Toda systems represent two well-known classes exhibiting drastically different solution structures. Over the years, existence results for these systems have highlighted discrepancies due to their unique solution structures. In this work, we aim to construct a monotone entropy form and establish the long-term existence of a flow of parabolic systems. As a result of our main theorem, we can prove existence results for some broad classes of elliptic systems, including both Liouville and Toda systems. The strength of our results is further underscored by the fact that no topological information about the Riemann surfaces is required and no positive lower bound of coefficient functions is postulated. 
 \end{abstract}

\maketitle

\numberwithin{equation}{section}
\allowdisplaybreaks

\section{Introduction}\label{sec1}
In this article we aim to study a broad class of second order elliptic system defined on a Riemann surface. Let $(M,g)$ be a Riemann surface with metric $g$, in this article we consider
\begin{equation} \label{Liouville}
\Delta u_i + \sum_{j=1}^{n} a_{ij} \left( \frac{h_j e^{u_j}}{\int h_j e^{u_j}} - 1, \right)= 0, \quad i=1,..,n
\end{equation}
where $\Delta$ is the Laplace-Beltrami operator ($-\Delta \ge 0)$, $h_1(x),...,h_n(x)$   are non-negative continuous functions not identically equal to zero, $A=(a_{ij})_{n\times n}$ is a constant matrix to be specified under different contexts later. The volume of $(M,g)$ is assumed to be $1$ for simplicity. Here we just mention that if all $a_{ij}$ are non-negative, the system (\ref{Liouville}) is called a Liouville system, if $A$ comes from some specific  Lie group, for example, if $A$ is the following Cartan matrix:
\[A=\left(\begin{array}{ccccc}
2 & -1 & 0 & ... & 0\\
-1 & 2 & -1 & ... & 0\\
\vdots & \vdots & \vdots & ... & \vdots \\
0 & 0 & 0 & ... & 2
\end{array}
\right),\]
system (\ref{Liouville}) is called a Toda system. 
Both
Liouville systems and Toda Systems have significant applications across various fields. In geometry, when either system reduces to a single equation ($n=1$), it generalizes the renowned Nirenberg problem, which has been extensively researched over the past few decades. In physics, Liouville systems emerge from the mean field limit of point vortexes in the Euler flow (see \cite{biler,wolansky1,wolansky2,yang}) and are intricately linked to self-dual condensate solutions of the Abelian Chern-Simons model with $N$ Higgs particles \cite{kimleelee,Wil}. In biology, they appear in the stationary solutions of the multi-species Patlak-Keller-Segel system \cite{wolansky3} and are important for studying chemotaxis \cite{childress}.   Toda system is a completely integrable system that is used in various fields including solid-state physics, mathematical physics, and even in the study of integrable systems and cluster algebras, etc (see \cite{takasaki,lin-wei-ye}).

Even though Liouville systems and Toda systems are both described by (\ref{Liouville}) with different coefficient matrices, they have drastically different structures of solutions. For example, Toda systems have discrete total integrals for global solutions \cite{lin-wei-ye} but Liouville systems have a continuum of energy (here we use energy to describe the total integration of global solutions). Solutions of Toda systems usually don't have radial symmetry, but global solutions of Liouville systems are radially symmetric in many cases \cite{CSW, lin-zhang-1}. Because of all these stark comparisons, there is barely any work that proves results for both of them. In this article we initiate a new approach to attack second order elliptic systems in general. By our innovative scheme we found we can combine both aforementioned systems in our new results and prove some existence results for a large class of elliptic systems.

Our assumption on the coefficient matrix $A$ is:
\begin{equation}\label{A-a-2}
A \mbox{ is symmetric, positive definite and the largest eigenvalue }<  8\pi.
\end{equation}
For the coefficient functions $h_i$ ($i=1,...,n$) we assume that 
\begin{equation}\label{h-a}
   h_i\ge 0,\quad \|h_i\|_{C^1(M)}<\infty,\quad h_i\not \equiv 0,\quad i=1,..,n.
\end{equation}
Under (\ref{A-a-2}) and (\ref{h-a}) we consider the following parabolic system:
\begin{equation} \label{Liouville system flow}
\left\{\begin{array}{ll}\partial_t u_i = \Delta u_i + \sum_{j=1}^{n} a_{ij} \left( \frac{h_j e^{u_j}}{\int h_j e^{u_j}} - 1 \right)\\
u_i(x,0)=u_{i,0}(x)\in C^{\infty}(M),\quad i=1,..,n,
\end{array}
\right.
\end{equation}
for $i=1, \ldots, n$ where we use $u_0(x)=(u_{1,0}(x),...,u_{n,0}(x))$ to denote the initial smooth function. 

Our main theorem is 
\begin{thm}\label{main-thm-1}
Let $A$ satisfy (\ref{A-a-2}), $h_1,...,h_n$ satisfy (\ref{h-a}) and $u_0$ be a smooth function on $M$, then (\ref{Liouville system flow}) has a unique global solution $u$ in $C([0,\infty),W^{1,2}(M))\cap C^{\infty}(M\times (0,\infty))$.
\end{thm}
The notation $u \in C([0,\infty),W^{1,2}(M)) \cap C^{\infty}(M \times (0,\infty))$ means for each $t\in [0,\infty)$, $u(t)\in W^{1,2}(M)$ and $\|u(t)\|_{W^{1,2}(M)}$ is continuous in $t$ and is $C^{\infty}$ in $(0,\infty)\times M$.

As a corollary of Theorem \ref{main-thm-1} we have the following existence result:
 \begin{cor}\label{main-thm-2}
Let $A$ satisfy (\ref{A-a-2}) and $h_1,...,h_n$ satisfy (\ref{h-a}), then (\ref{Liouville}) has a solution. 
\end{cor}

Here we make a few remarks about Corollary \ref{main-thm-2}. Firstly if $A$ is a nonnegative matrix, the system is a Liouville system. Corollary \ref{main-thm-2} is the first existence theorem for Liouville systems that assumes the coefficient matrix to be positive definite. In comparison, the existence theorems of Lin-Zhang \cite{lin-zhang-2} and Gu-Zhang \cite{gu-zhang} require negative eigenvalues on $A$. Secondly, there is no requirement on the topology of the manifold $(M,g)$, while in previous results \cite{lin-zhang-1, lin-zhang-2, gu-zhang,gu-zhang-2}, the topology of the manifold is required to be nontrivial. Thirdly, the coefficient functions $h_i$ are not required to be bounded below by positive constants. No existence results or blowup analysis have appeared before with such weak assumptions. It is also clear that the assumption of $A$ in (\ref{A-a-2}) also includes all Toda systems with coefficient matrices as Cartan matrices $A_n$. In this sense Corollary \ref{main-thm-2} unifies the two drastically different elliptic systems. As far as we know before Corollary \ref{main-thm-2} there had never been a theorem that proves existence of solutions for both Liouvlle systems and Toda systems.

As mentioned before we normalize the volume $\int_{M} 1 = 1$ for simplicity. 
This implies that the solution of \eqref{Liouville system flow} satisfies $\int_{M} \partial_t u_i dx = 0$.
Therefore, $\int_{M} u_i$ is a constant.
We may assume $\int_{M} u_{i,0} = 0$, then we get
\[
\int_{M} u_i = 0.
\]
We denote $A^{-1} = (a^{ij})$ and $u^i = \sum_j a^{ij}u_j$.
Throughout this paper, we mainly write integral and derivatives with respect to $g$ unless otherwise specified. 

The organization of the article is as follows: In section two we list some preliminary tools for the proof of short and long time existence of the flow. In particular, Lemma \ref{lem Struwe}, which can be found in \cite{borer,struwe-2}, plays a crucial role in the proof of Theorem \ref{main-thm-1}. In section three we prove the short time existence by a fix point argument. Finally in section four we prove the long time existence by a carefully crafted Moser iteration.

\section{Preliminary}
\label{sec2}


We define entropy of \eqref{Liouville system flow} by
\begin{equation} \label{entropy positive}
\begin{split}
K(t) =& \frac{1}{2} \sum_{i,j} \int_{M} a^{ij} \nabla u_i \nabla u_j - \sum_{j}\log \left(\int_{M} h_j e^{u_j} \right) 
\end{split}
\end{equation}
where $a^{ij}$ are entries of $A^{-1}$. The following lemma gives the monotonicity of $K$:

\begin{lem} \label{lem K dec}
Let $(u_i)$ be a smooth solution of \eqref{Liouville system flow} on $M \times [0,T]$.
Also assume $A$ is positive definite.
Then the entropy $K(t)$ is non-increasing.
\end{lem}

\begin{proof}
From the equation, we obtain that
\[
\begin{split}
K'(t) =& \sum_{i,j} \int_{M} a^{ij} \nabla u_i \partial_t (\nabla u_j) - \sum_{k} \int_{M} \frac{h_k e^{u_k}}{\int h_k e^{u_k}} \partial_t u_k \\
=& -\sum_{i,j} \int_{M} a^{ij} \left( \Delta u_i + \sum_{k} a_{ik} \left( \frac{h_k e^{u_k}}{\int h_k e^{u_k}} - 1 \right) \right) \partial_t u_j\\
=& - \sum_{i,j} \int_{M} a^{ij} \partial_t u_i \partial_t u_j \leq0
\end{split}
\]
if $A$ is positive definite. Lemma \ref{lem K dec} is established. 
\end{proof}

The following theorem provides an estimate for $\int_M e^{u^2}$.
And using this, we can obtain an estimate for $\int_M e^u$.
See for example Borer-Elbau-Weth \cite{borer} Lemma 2.1 or Struwe \cite{struwe-2} Theorem 2.2.

\begin{thm} \label{thm Struwe}
    Let $M$ be a closed and orientable surface.
    Then for any $\beta < 4\pi$,
    \[
    C_{TM}(\beta) := \sup \left\{\int_{M} e^{u^2} ; u \in W^{1,2}(M), \|\nabla u\|_{L^2}^2 \leq \beta, \bar{u} = 0 \right\} < \infty.
    \]
\end{thm}

Using Young's inequality
\[
2 |p (u-\bar{u}) | \leq \frac{\beta (u-\bar{u})^2}{\|\nabla u\|_{L^2}^2} + \frac{p^2}{\beta} \|\nabla u\|_{L^2}^2
\]
we can conclude the following lemma.

\begin{lem} \label{lem Struwe}
For any $u \in W^{1,2}(M)$ and for any $p \in \mathbb{R}$ and $\beta < 4\pi$,
\[
\frac{p^2}{\beta} \int_{M}  |\nabla u|^2  \geq \log \left( \frac{1}{C_{TM}(\beta)} \int_{M} e^{2p(u-\bar{u})}\right)
\]
where $C_{TM}(\beta) < \infty$ is a positive constant.
\end{lem}

Now we denote $\frac{1}{8\pi} <\lambda \leq \Lambda$ such that for any $\xi \in \mathbb{R}^n$,
\begin{equation} \label{eq cond A^-1}
\lambda |\xi|^2 \leq A^{-1}(\xi,\xi) \leq \Lambda |\xi|^2.
\end{equation}
Fix $\beta = \beta(\lambda) < 4\pi$ such that
\begin{equation} \label{eq beta}
\frac{1}{8\pi} < \frac{1}{2\beta} < \lambda.
\end{equation}
Then the following lemma gives a lower bound for all $K(t)$.

\begin{lem} \label{lem lower bd of K}
Let $(u_i)$ be a smooth solution of \eqref{Liouville system flow} on $M \times [0,T]$.
Also, assume \eqref{eq cond A^-1}.
Then
\[
K(t) \geq \frac{\lambda - \frac{1}{2\beta}}{2} \sum_{j} \int_{M} |\nabla u_j|^2-\sum_{j} \log \left( C_{TM}(\beta) M_j \right).
\]
\end{lem}

\begin{proof}
By direct computation and Lemma \ref{lem Struwe} with $p=\frac{1}{2}$ and \eqref{eq cond A^-1},
\[
\begin{split}
    K(t) \geq& \frac{1}{2} \sum_{i,j} \int_{M} a^{ij} \nabla u_i \nabla u_j - \sum_{j}\log \left( C_{TM}(\beta) \max_{M} h_j \right) - \sum_{j} \log \left( \frac{1}{C_{TM}(\beta)} \int_{M} e^{u_j} \right) \\
    \geq& \frac{1}{2} \sum_{i,j} \int_{M} a^{ij} \nabla u_i \nabla u_j - \sum_{j}\log \left( C_{TM}(\beta) M_j \right) - \sum_{j} \frac{1}{4 \beta} \int_{M} |\nabla u_j|^2\\
    \geq& \frac{\lambda - \frac{1}{2\beta}}{2} \sum_{j} \int_{M} |\nabla u_j|^2 - \sum_{j}\log \left( C_{TM} (\beta) M_j \right).
\end{split}
\]

\end{proof}

From Lemma \ref{lem lower bd of K} we have
\begin{equation} \label{eq K lower bd}
    K(t) \geq c_0^{-1} \sum_{j} \int_{M} |\nabla u_j|^2 - C_M
\end{equation}
where $c_0$ and $C_M$ are positive constants independent of $t$.

A consequence of Lemma \ref{lem K dec} and \eqref{eq K lower bd} is that there exists the limit 
\[
K(\infty) := \lim_{t \to \infty} K(t) \geq -C_M.
\]
Also, we have that for any $t \in [0,T]$,
\begin{equation} \label{eq nabla u bdd}
\begin{split}
\sum_{j}\int_{M} |\nabla u_j|^2(t) \leq& c_0 \left( K(0) + C_M\right)
=: C_0 < +\infty
\end{split}
\end{equation}
where $C_0>0$ depends only on $\lambda, n, \beta, M_j$ and $u_{j,0}$.

\section{Short-time existence}
\label{sec3}

In this section, we show the short-time existence.
We first introduce some necessary notation.

Let $\Omega \subset \mathbb{R}^n$ and denote $\Omega_T = \Omega \times (0,T)$ for $T>0$.
For $k$ to be a nonnegative integer, $1 \leq p < \infty$, we define
\[
W^{2k,k}_{p}(\Omega_T) := \{ u \in L^p(\Omega_T) : \|u\|_{W^{2k,k}_{p}(\Omega_T)} < \infty\}
\]
where
\[
\|u\|_{W^{2k,k}_{p}(\Omega_T)} := \left( \iint_{\Omega_T}  \sum_{|\alpha| + 2r \leq 2k} |D^\alpha D_t^r u |^p dxdt \right)^{1/p}.
\]
For $1 \leq p,q \leq \infty$, we define
\[
L^q(L^p(\Omega)) := L^q([0,T];L^p(\Omega)) = \{u : \int_{0}^{T} \|u(t)\|_{L^p(\Omega)}^q < \infty\}
\]
with the norm
\[
\|u\|_{L^q(L^p(\Omega))} = \left(\int_{0}^{T} \|u(t)\|_{L^p(\Omega)}^q \right)^{1/q}.
\]
We also define
\[
C^{\alpha,\alpha/2}(\Omega_T) := \{u : |u|_{C^{\alpha,\alpha/2}(\Omega_T)} < \infty\}
\]
where
\[
\begin{split}
|u|_{C^{\alpha,\alpha/2}(\Omega_T)} =& \sup_{\Omega_T} |u| + [u]_{C^{\alpha,\alpha/2}(\Omega_T)},\\
[u]_{C^{\alpha,\alpha/2}(\Omega_T)} =& \sup_{\substack{(x,t),(y,s) \in \Omega_T \\ (x,t) \neq (y,s)}} \frac{|u(x,t)-u(y,s)|}{(|x-y|^2+|t-s|)^{\frac{\alpha}{2}}}
\end{split}
\]
and
\[
C^{2k+\alpha,k+\alpha/2}(\Omega_T) := \{u : D^\beta D_t^r u \in C^{\alpha,\alpha/2}(\Omega_T) \text{ for any } \beta,r \text{ such that } |\beta| +2r \leq 2k \}.
\]

Now we have the following version of Sobolev embedding.
Let $\Omega \subset \mathbb{R}^2$ and $\nabla$ denotes spatial derivative.

\begin{lem}
(Sobolev embedding of $t$-anisotropic functions)
[\cite{wu-zhuoqun} Theorem 1.4.1 or \cite{borer} Theorem 2.2 and Theorem 3.13]

Let $u \in W^{2,1}_p (\Omega_T)$, $p>2$.
Then
\[
|u|_{C^{\alpha,\alpha/2}(\Omega_T)} \leq C_1(\Omega_T) \|u\|_{W^{2,1}_{p}(\Omega_T)} 
\]
with $0<\alpha<2-\frac{4}{p}$.
Also, we have
\[
\begin{cases}
    \|\nabla u\|_{L^\infty(L^q(\Omega))} &\leq C_2(\Omega_T) \|u\|_{W^{2,1}_p (\Omega_T)} \quad   \qquad \text{ if }\ \ p < 4, \ q \leq \frac{2p}{4-p} \\
   \|\nabla u\|_{L^\infty(L^q(\Omega))} &\leq C_2(\Omega_T) \|u\|_{W^{2,1}_p (\Omega_T)} \quad  \qquad \text{ if }\ \ p = 4 , \ q < \infty\\
   \|\nabla u\|_{L^\infty(C^\alpha(\Omega))} &\leq C_2(\Omega_T) \|u\|_{W^{2,1}_p (\Omega_T)} \quad  \qquad \text{ if }\ \ p > 4, \ \alpha = 1 - \frac{4}{p}.
\end{cases}
\]
\end{lem}

In particular, we have that for any $u \in W^{2,1}_p (M_T)$,
\[
\int_{M} |\nabla u|^2(t) \leq C_2 \|u\|_{W^{2,1}_p(M_T)}
\]
where $C_2 = C_2(M_T)$ and $M_T = M \times (0,T)$.

Next, we need the following existence theorem.
Fix $T_0>0$ and consider $0 < T \leq T_0$.

\begin{prop} \label{prop BEW}
[\cite{borer} Proposition 6.2]
Let $u_0 \in W^{2,p}(M)$ and $f \in L^p(M_T)$.
Then there exists a unique strong solution $u \in W^{2,1}_p(M_T)$ of the initial value problem
\begin{equation} \label{eq IVP}
\begin{cases}
    \partial_t u &= \Delta u + f \qquad \quad \text{ on } M_T \\
    u(x,0) &= u_0(x) \qquad \qquad \text{ in } M
\end{cases}
\end{equation}
satisfying
\begin{equation} \label{eq sol Lp}
    \|u\|_{W^{2,1}_p(M_T)} \leq C_3 \left(\|u_0\|_{W^{2,p}(M)} + \|f\|_{L^p(M_T)} \right)
\end{equation}
for some constant $C_3$ which depends on $T_0$ but not on $T$.
Moreover, if $f \in C^{\alpha}(M_T)$ for some $\alpha>0$, then $u \in C(\overline{M_T}) \cap C^{2,1}(M_T)$ and
\[
\|u_0\|_{W^{1,2}(M)} \geq \limsup_{t \to 0^+} \|u(t)\|_{W^{1,2}(M)}.
\]
\end{prop}

Another, more commonly used version of above proposition is the following.
\begin{prop}
    Let $u_0 \in C^{2,\alpha}(M)$ and $f \in C^{\alpha,\alpha/2}(M_T)$.
    Then there exists a unique strong solution $u \in C^{2+\alpha,1+\alpha/2}(M_T)$ of the initial value problem \eqref{eq IVP} satisfying
    \begin{equation} \label{eq sol Holder}
        \|u\|_{C^{2+\alpha,1+\alpha/2}(M_T)} \leq C_3 \left(\|u_0\|_{C^{2,\alpha}(M)} + \|f\|_{C^{\alpha,\alpha/2}(M_T)} \right)
    \end{equation}
    for some constant $C_3$ which depends on $T_0$ but not on $T$.
\end{prop}

Now we set up the Banach space $W^{2,1}_p(M_T)$ and for any $R>0$, its closed subset
\[
X_{R,i} := \left\{u \in W^{2,1}_p(M_T) : \|u\|_{W^{2,1}_p(M_T)}\leq R, u(x,0)=u_{i,0}(x), \int_{M} u = 0\right\}.
\]
Denote $X_R = \prod_{i=1}^{n} X_{R,i}$.
Then $X_R$ is a closed subset of the Banach space 
\[X = \prod_{i=1}^{n} W^{2,1}_p (M_T)\] which has a norm given by
\[
\|u\|_{X} = \sum_{i=1}^{n} \|u_i\|_{W^{2,1}_p(M_T)}.
\]

Fix $R$ such that
\[
2C_3 \max_{i} \|u_{i,0}\|_{W^{2,p}(M)} \leq R.
\]
Then we define the map $\Phi : X \to X$ as follows: for $v=(v_1,...,v_n) \in \prod X_{R,i}$, 
\[\Phi(v) = u = (u_1,...,u_n)\] where $u_i$ is the unique solution of
\begin{equation} \label{eq IVP2}
    \begin{cases}
    \partial_t u_i &= \Delta u_i + f_i \qquad \quad \text{ on } M_T \\
    u_i(x,0) &= u_{i,0}(x) \qquad \qquad \text{ in } M
\end{cases}
\end{equation}
for
\begin{equation} \label{eq f_i}
f_i = \sum_{j} a_{ij} \left(\frac{h_j e^{v_j}}{\int h_j e^{v_j}} - 1 \right).
\end{equation}

Fix a positive constant $q>1$.

\begin{lem} \label{lem ball}
    The map $\Phi$ defined above restricts to $\Phi : X_R \to X_R$ if
    \begin{equation}
        T \leq \frac{R}{2C_3 C |A|^p \sum_{j}  \left( \frac{M_j^p}{\|h_j^{1/q}\|_{L^{1}}^{pq}} \left( C_{TM} e^{\frac{1}{(q-1)^2 8\pi} C_2 R}\right)^{(q-1)p}  C_{TM} e^{\frac{p^2}{8\pi} C_2 R} + 1 \right)}.
    \end{equation}
\end{lem}

\begin{proof}
We first show that $\Phi$ maps $X_R$ to $X_R$ if $T$ is small enough.
Let $u_i$ be a solution of \eqref{eq IVP2}.
Then clearly $u_i(x,0) = u_{i,0}(x)$ and $\int_{M} \partial_t u_i = 0$, which implies $\int_{M}  u_i = 0$ under the assumption $\int_{M}  u_{i,0} = 0$.
It remains to show that $\|u_i\|_{W^{2,1}_p(M_T)} \leq R$ for all $T$ small enough if $v_j \in X_R$.

Since $v_j \in X_{R,j}$, we have $\int_{M} |\nabla v_j|^2(t) \leq C_2 \|v_j\|_{W^{2,1}_p(M_T)} \leq C_2 R$.
Also note that by Lemma \ref{lem Struwe},
\[
\begin{split}
    0<&\|h_j^{1/q}\|_{L^{1}} = \int h_j^{1/q} \leq \left( \int h_j e^{v_j}\right)^{\frac{1}{q}} \left( \int e^{\frac{-1}{q-1}v_j} \right)^{1-\frac{1}{q}}\\
    \leq& \left( \int h_j e^{v_j} \right)^{\frac{1}{q}} \left( C_{TM} e^{\frac{1}{(q-1)^2 8\pi} C_2 R} \right)^{1-\frac{1}{q}}
\end{split}
\]
which implies
\begin{equation} \label{eq h e^v lower}
    \int h_j e^{v_j} \geq \|h_j^{1/q}\|_{L^{1}}^{q} \left( C_{TM} e^{\frac{1}{(q-1)^2 8\pi} C_2 R} \right)^{-q+1} > 0.
\end{equation}
Then
\[
\begin{split}
   \iint |f_i|^p \leq& \iint \sum_{j} |A|^p \left| \frac{h_j e^{v_j}}{\int h_j e^{v_j}} - 1\right|^p\\
    \leq& C |A|^p \sum_j  \iint \left( \frac{h_j^p e^{pv_j}}{(\int h_j e^{v_j})^p} + 1\right)\\
    \leq& C |A|^p \sum_j \left( \int_{0}^{T} \frac{M_j^p}{\|h_j^{1/q}\|_{L^{1}}^{pq}} \left( C_{TM} e^{\frac{1}{(q-1)^2 8\pi} C_2 R}\right)^{(q-1)p} \int_M e^{pv_j} + T \right)\\
    \leq& C |A|^p \sum_j \left( \frac{M_j^p}{\|h_j^{1/q}\|_{L^{1}}^{pq}} \left( C_{TM} e^{\frac{1}{(q-1)^2 8\pi} C_2 R}\right)^{(q-1)p}  C_{TM} e^{\frac{p^2}{8\pi} C_2 R} + 1 \right) T\\
    \leq& \frac{R}{2C_3}.
\end{split}
\]

Here we apply Lemma \ref{lem Struwe} and denote $|A| = \max |a_{ij}|$.
Finally, by \eqref{eq sol Lp}, we have
\[
    \|u_i\|_{W^{2,1}_p(M_T)} \leq C_3 \left(\|u_{i,0}\|_{W^{2,p}(M)} + \|f_i\|_{L^p(M_T)} \right) \leq R.
\]
This completes the proof.
\end{proof}

\begin{thm} \label{thm short time}
(Short-time existence)
Let $p>2$ and $u_{i,0} \in W^{2,p}(M)$ with $\int_{M} u_{i,0}=0$.
    For $T$ satisfying
    \begin{equation} \label{eq cond T}
    \begin{split}
    T \leq& \min \left\{ \frac{R}{2C_3 C |A|^p \sum_{j}  \left( \frac{M_j^p}{\|h_j^{1/q}\|_{L^{1}}^{pq}} \left( C_{TM} e^{\frac{1}{(q-1)^2 8\pi} C_2 R}\right)^{(q-1)p}  C_{TM} e^{\frac{p^2}{8\pi} C_2 R} + 1 \right)}, \right.\\
    & \left. \qquad \qquad\frac{1}{4n C_3 C |A|^p \sum_{j} \frac{M_j^{2p}}{\|h_j^{1/q}\|_{L^{1}}^{2pq}} \left( C_{TM} e^{\frac{1}{(q-1)^2 8\pi} C_2 R}\right)^{(q-1)2p}  C_{TM} e^{\frac{p^2}{8\pi} C_2 R} e^{pC_1 R}} \right\},
    \end{split}
    \end{equation}
    we have
    \[
    \|\Phi(v) - \Phi(\tilde{v})\|_X \leq \frac{1}{2} \|v - \tilde{v}\|_{X}.
    \]    
    Hence, by Banach fixed point theorem, for $T$ small enough, there exists a unique fixed point $u \in X_R$ such that $\Phi(u)=u$, that is, $u = (u_i)$ solves \eqref{Liouville system flow} with the initial condition $u_i(x,0) = u_{i,0}(x)$.
\end{thm}

\begin{proof}
    Let $u = \Phi(v)$, $\tilde{u} = \Phi(\tilde{v})$.
    Also denote
    \[
    f_i = \sum_{j} a_{ij} \left(\frac{h_j e^{v_j}}{\int h_j e^{v_j}} - 1 \right), \quad \tilde{f}_i = \sum_{j} a_{ij} \left(\frac{h_j e^{\tilde{v}_j}}{\int h_j e^{\tilde{v}_j}} - 1 \right).
    \]
    Then $u_i-\tilde{u}_i$ solves
    \[
    \partial_t (u_i - \tilde{u}_i) = \Delta (u_i - \tilde{u}_i) + (f_i - \tilde{f}_i)
    \]
    with the initial condition $(u_i-\tilde{u}_i)(x,0) = 0$.
    As in the proof of Lemma \ref{lem ball}, $f_i-\tilde{f}_i \in L^p(M_T)$.
    Then by \eqref{eq sol Lp}, we have
    \[
    \|u-\tilde{u}\|_X = \sum_{i=1}^{n} \|u_i - \tilde{u}_i\|_{W^{2,1}_p(M_T)} \leq C_3 \sum_{i=1}^{n} \|f_i - \tilde{f}_i\|_{L^p(M_T)}.
    \]
    To estimate $\|f_i-\tilde{f}_i\|_{L^p(M_T)}$, note that
    \[
    \begin{split}
    f_i -\tilde{f}_i =& \sum_{j} a_{ij} \left(\frac{h_j e^{v_j}}{\int h_j e^{v_j}} - \frac{h_j e^{\tilde{v}_j}}{\int h_j e^{\tilde{v}_j}} \right)\\
    =& \sum_{j} a_{ij} \frac{1}{\int h_j e^{v_j} \int h_j e^{\tilde{v}_j}} \left(  h_j e^{v_j} \int_M h_j e^{\tilde{v}_j} -  h_j e^{\tilde{v}_j} \int_M h_j e^{v_j} \right).
    \end{split}
    \]
    Using
    \[
    \begin{split}
    \left| f \int g - g \int f\right|^p =& \left| f \int (g-f) - (g-f) \int f \right|^p\\
    \leq& C \left( |f|^p  \int |g-f|^p + |g-f|^p \int |f| ^p\right)
    \end{split}
    \]
    and by \eqref{eq h e^v lower}, we have
    \[
    \begin{split}
        \iint |f_i - \tilde{f}_i|^p \leq& |A|^p \sum_{j} \frac{\left( C_{TM} e^{\frac{1}{(q-1)^2 8\pi} C_2 R}\right)^{(q-1)2p}}{\|h_j^{1/q}\|_{L^{1}}^{2pq}}  \iint \left|  h_j e^{v_j} \int_M h_j e^{\tilde{v}_j} -  h_j e^{\tilde{v}_j} \int_M h_j e^{v_j} \right|^p\\
        \leq& C |A|^p \sum_{j}\frac{\left( C_{TM} e^{\frac{1}{(q-1)^2 8\pi} C_2 R}\right)^{(q-1)2p}}{\|h_j^{1/q}\|_{L^{1}}^{2pq}} \\
        & \cdot\iint \left( (h_j)^p e^{pv_j} \int_{M} \left|h_j e^{\tilde{v}_j} - h_j e^{v_j}\right|^p + |h_j e^{\tilde{v}_j} - h_j e^{v_j}|^p  \int_{M} (h_j)^p e^{pv_j} \right)\\
        \leq& 2C |A|^p \sum_{j} \frac{\left( C_{TM} e^{\frac{1}{(q-1)^2 8\pi} C_2 R}\right)^{(q-1)2p} M_j^{2p}}{\|h_j^{1/q}\|_{L^{1}}^{2pq}} \int_{0}^{T} \int_{M} |e^{\tilde{v}_j} - e^{v_j}|^p \int_{M} e^{pv_j}\\
        \leq& 2C |A|^p \sum_{j} \frac{\left( C_{TM} e^{\frac{1}{(q-1)^2 8\pi} C_2 R}\right)^{(q-1)2p} M_j^{2p}}{\|h_j^{1/q}\|_{L^{1}}^{2pq}}  \int_{0}^{T} C_{TM} e^{\frac{p^2}{8\pi} C_2 R}\int_{M} |e^{\tilde{v}_j} - e^{v_j}|^p.
    \end{split}
    \]
    Finally, we note that
    \[
    \left| e^{v_j} - e^{\tilde{v}_j} \right| \leq e^{\max \{v_j, \tilde{v}_j\}} \left( 1 - e^{-|v_j - \tilde{v}_j|} \right) \leq e^{C_1 R} |v_j - \tilde{v}_j|.
    \]
    Therefore, we get
    \[
    \begin{split}
    \iint |f_i - \tilde{f}_i|^p 
    \leq & 2C |A|^p \sum_{j} \frac{ M_j^{2p}}{\|h_j^{1/q}\|_{L^{1}}^{2pq}} \left( C_{TM} e^{\frac{1}{(q-1)^2 8\pi} C_2 R}\right)^{(q-1)2p} \\
    &\cdot C_{TM} e^{\frac{p^2}{8\pi} C_2 R} e^{pC_1 R} \|v_j - \tilde{v}_j\|_{L^p(M_T)} T\\
    \leq& \frac{1}{2nC_3} \|v - \tilde{v}\|_X.
    \end{split}
    \]
    This gives $\|\Phi(v) - \Phi(\tilde{v})\|_X = \|u-\tilde{u}\|_{X} \leq \frac{1}{2} \|v-\tilde{v}\|_X$ as desired.
\end{proof}

Next, we show that the above obtained solution is in fact smooth under suitable conditions for $h_j$ and $u_{i,0}$.
Note that $u_i \in W^{2,1}_p(M_T)$ for $p>2$ implies that $u_i \in C^{\alpha,\alpha/2}(M_T)$ for any $0<\alpha < 2-\frac{4}{p}$.

\begin{lem}
Let $u_j \in C^{\alpha,\alpha/2}(M_T)$ and $h_j \in C^{\alpha}(M)$ for all $j=1, \ldots, n$.
Then $f_i \in C^{\alpha,\alpha/2}(M_T)$ where $f_i$ is given in \eqref{eq f_i}.
\end{lem}

\begin{proof}
From the assumption for $u_j$, we let $R := \max_j \sup_{M_T} |u_j(x,t)|$.
First note that
\begin{equation} \label{eq h e^v lower2}
\int h_j e^{u_j} \geq \|h_j^{1/q}\|_{L^{1}}^{q} \left(\int e^{\frac{-1}{q-1} u_j} \right)^{-q+1} \geq \|h_j^{1/q}\|_{L^{1}}^{q} e^{-R}.
\end{equation}
Also note that for any $x,y \in M$ with $x \neq y$ and for any $t \in (0,T)$,
\[
\begin{split}
    |f_i(x,t)-f_i(y,t)| =& \sum_j \frac{a_{ij}}{\int h_j e^{u_j}(t)} \left| h_j(x) e^{u_j(x,t)} - h_j(y) e^{u_j(y,t)}\right|\\
    \leq& \sum_j \frac{a_{ij}}{\int h_j e^{u_j}(t)} \left( |h_j(x)-h_j(y)| e^{u_j(x,t)} + h_j(y) \left|e^{u_j(x,t)}-e^{u_j(y,t)} \right|\right)\\
    \leq& \sum_j \frac{a_{ij}e^{R}}{\|h_j^{1/q}\|_{L^{1}}^{q}} \left( e^R |h_j(x)-h_j(y)| + M_j e^R \left|u_j(x,t)- u_j(y,t) \right|\right)\\
    \leq& C |x-y|^{\alpha}
\end{split}
\]
because $h_j(\cdot),u_j(\cdot,t)$ are H{\"o}lder continuous.
Similarly, for any $t,s \in (0,T)$ with $t \neq s$ and for any $x \in M$,
\[
\begin{split}
    &|f_i(x,t)-f_i(x,s)|\\
    =& \sum_j \frac{a_{ij} h_j(x)}{\int h_j e^{u_j}(t) \int h_j e^{u_j}(s)} \left| \int h_j e^{u_j} (s) \cdot e^{u_j(x,t)} - \int h_j e^{u_j} (t) \cdot e^{u_j(x,s)}\right|\\
    \leq& \sum_j a_{ij}  \frac{M_j^2 e^{2R}}{\|h_j^{1/q}\|_{L^{1}}^{2q}} \left( \int_M |e^{u_j(s)} - e^{u_j(t)}| \cdot e^{u_j(x,t)} + \left| e^{u_j(x,t)} - e^{u_j(x,s)}  \right| \int_M e^{u_j}(t) \right)\\
    \leq& \sum_j a_{ij} \frac{M_j^2 e^{2R}}{\|h_j^{1/q}\|_{L^{1}}^{2q}} e^{2R} \left( \int_M |u_j(y,t)-u_j(y,s)| dy + |u_j(x,t)-u_j(x,s)| \right)\\
    \leq& C |t-s|^{\alpha/2}
\end{split}
\]
because $u_j(x,\cdot)$ is H{\"o}lder continuous.

This shows $f_i \in C^{\alpha,\alpha/2}(M_T)$ and the proof is complete.
\end{proof}

Therefore, by \eqref{eq sol Holder}, $u_i \in C^{2+\alpha,1+\alpha/2}(M_T)$ if $u_{i,0} \in C^{2,\alpha}(M)$ and $h_j \in C^{\alpha}(M)$.
This implies $f_i \in C^{2+\alpha,1+\alpha/2}(M_T)$.
By Schauder estimate and bootstrapping argument, 
we can conclude that $u_i \in C^\infty(M_T)$ if $h_j$ are smooth.

Finally, we show that $u \in C([0,T),W^{1,2}(M))$.
Set
\[E_i(t) = \frac{1}{2} \int_{M} |\nabla u_i(t)|^2\] for $t \in (0,T)$.
Note that for any $0<t_1<t_2<T$,
\[
\begin{split}
     (E_i(t_2) - E_i(t_1)) =& \frac{1}{2} \int_{t_1}^{t_2} \left( \partial_t \int_{M} |\nabla u_i(t)|^2 \right) dt\\
     =& \int_{t_1}^{t_2} \int_{M} \nabla u_i(t) \nabla \partial_t u_i(t)\\
     =& -\int_{t_1}^{t_2} \int_{M} |\partial_t u_i(t)|^2 + \int_{t_1}^{t_2} \int_{M}f_i \partial_t u_i(t)\\
     \leq& -\frac{1}{2} \int_{t_1}^{t_2} \int_{M} |\partial_t u_i(t)|^2 + \frac{1}{2} \int_{t_1}^{t_2} \int_{M} |f_i|^2\\
     \leq& C (t_2-t_1)
\end{split}
\]
where the constant $C$ above depends on $|A| = \max a_{ij}$, $M_j$, $\| h_j\|_{L^{1/q}}$, $q$, $C_{TM}$, $C_2$ and $\sup_{M_T} u_j$, from the proof of Lemma \ref{lem ball}.
Hence $E_i(t)$ is uniformly continuous on $(0,T)$ and is therefore bounded on $(0,T)$.

Now by contradiction, assume that for some $i$, $u_i(t)$ is not continuous at $t=0$ in $W^{1,2}(M)$ norm.
Then there exists $t_n \to 0$ and $\varepsilon>0$ such that
\begin{equation} \label{eq u_i not conti}
\|u_i(t_n)-u_{i,0}\|_{W^{1,2}(M)} \geq \varepsilon.
\end{equation}
Since $E_i(t)$ is bounded, we can find a subsequence, still denoted by $(t_n)$, such that $u_i(t_n)$ converges weakly in $W^{1,2}(M)$, which implies that $u_i(t_n)$ converges strongly in $L^2(M)$.
This limit is $u_{i,0}$ and so we have $u_i(t_n) \to u_{i,0}$ weakly in $W^{1,2}(M)$.
By Proposition \ref{prop BEW} and lower semicontinuity, we get
\[
\limsup_{n \to \infty} \|u_i(t_n)\|_{W^{1,2}(M)} \leq \|u_{i,0}\|_{W^{1,2}(M)} \leq \liminf_{n \to \infty} \|u_i(t_n)\|_{W^{1,2}(M)}
\]
which implies $\|u_i(t_n)\|_{W^{1,2}(M)} \to \|u_{i,0}\|_{W^{1,2}(M)}$ and hence $u_i(t_n) \to u_{i,0}$ strongly in $W^{1,2}(M)$.
This contradicts \eqref{eq u_i not conti}.

In summary, we obtain

\begin{prop}
    Let $u_{i,0} \in W^{2,p}(M)$ with $\int_{M} u_{i,0} = 0$ and $h_j$ be smooth on $M$.
    Then the solution $u_i$ obtained in \ref{thm short time} is smooth on $M_T$.
    Moreover, $u_i \in C([0,T),W^{1,2}(M))$.
\end{prop}

\section{Global existence}
\label{sec4}

In this section, we show the global existence.
The result is not direct.
In fact, we can easily show the uniform lower bound, while showing uniform upper bound is much more difficult.

For example, we can let
\[
X(x,t) = \sum_i e^{u_i}>0.
\]
From \eqref{eq nabla u bdd}, we have
\[
\sum_j\int_M |\nabla u_j|^2(t) \leq  C_0 < \infty.
\]
Now, using \eqref{eq h e^v lower} and replacing $C_2 R$ by $C_0 $, we can estimate the equation by 
\[
\begin{split}
\partial_t u_i =& \Delta u_i + \sum_{j} a_{ij} \left( \frac{h_j e^{u_j} - \int h_j e^{u_j}}{\int h_j e^{u_j}} \right)\\
\leq& \Delta u_i + |A| \sum_j \frac{ h_j e^{u_j} + \int h_j e^{u_j}}{\|h_j^{1/q}\|_{L^{1}}^{q}} \left( C_{TM} e^{\frac{1}{(q-1)^2 8\pi} C_0}\right)^{q-1} \\
\leq& \Delta u_i + |A| \sum_j \frac{M_j}{\|h_j^{1/q}\|_{L^{1}}^{q}} \left( C_{TM} e^{\frac{1}{(q-1)^2 8\pi} C_0}\right)^{q-1} \sum_j\left( e^{u_j} + \int e^{u_j} \right).
\end{split}
\]
Multiplying with $e^{u_i}$ and sum with $i$, we have
\[
X_t \leq \Delta X + \left( |A| \sum_j \frac{M_j}{\| h_j^{1/q}\|_{L^{1}}^{q}} \left( C_{TM} e^{\frac{1}{(q-1)^2 8\pi}  C_0}\right)^{q-1}  \right) \left( X^2 + X \int X \right).
\]
But this inequality does not lead to the uniform boundedness due to the square growth in the RHS.
Recall that the harmonic map flow satisfies stronger inequality
\[
|u_t - \Delta u| = |A(du,du)| \leq C |du|^2
\]
and may develop finite time blow up, see Struwe \cite{struwe-1} or Chang-Ding-Ye \cite{chang-ding-ye}.

In our case, however, we can obtain a uniform boundedness.
To get the result, we first describe the boundedness of $u_i$ at some time $t_T$.
Then by using Moser iteration, we obtain uniform boundedness of $u_i$.
Together with the uniqueness property, we will get global existence for $u_i$.

\begin{lem} \label{lem unif bound-T}
    Let $u_i$ be a solution of \eqref{eq IVP2} on $M \times [0,T_0]$.
    Assume $T_0 \geq 1$.
    For any $0<T \leq  T_0-1 $, there exists $t_T \in [T,T+1)$ and a constant $C_4>0$ independent on $T$ such that
    \begin{equation}
        \|u_i(t_T)\|_{L^{\infty}(M)} \leq C_4.
    \end{equation}
\end{lem}

\begin{proof}
Together with \eqref{eq nabla u bdd} and Lemma \ref{lem Struwe}, and the fact $\bar{u_i} = \int_{M} u_i = 0$, for any $p>0$, we get
\begin{equation} \label{eq e^pu}
\int_{M} e^{pu_i}(t) \leq C_{TM} e^{\frac{p^2}{8\pi} C_0}.
\end{equation}
Also, by Poincar{\'e} inequality, we have that for any $t$,
\begin{equation} \label{eq poincare}
    \sum_{i} \|u_i(t)\|_{L^2(M)}^2 = \sum_{i} \|u_i - \bar{u}_i\|_{L^2(M)}^2 \leq C \sum_{i} \|\nabla u_i(t)\|_{L^2(M)}^2 \leq C  C_0.
\end{equation}

Next, from \eqref{eq cond A^-1} and Lemma \ref{lem K dec}, for any $T$,
\[
\begin{split}
\sum_{i}\int_{0}^{T} \int_{M} |\partial_t u_i |^2 \leq& \int_{0}^{T} \lambda^{-1} \sum_{i,j} \int_{M} a^{ij} \partial_t u_i \partial_t u_j\\
=& \lambda^{-1} \left( K(0) - K(T) \right)\\
\leq& \lambda^{-1} \left( K(0) +  C_M \right)\\
=& \lambda^{-1} c_0^{-1} C_0<\infty.
\end{split}
\]
Then for any $T$ with $0 < T \leq T_0 -1$, there exists $t_T \in [T,T+1)$ such that
\[
\begin{split}
\sum_{i} \int_{M} |\partial_t u_i|^2(t_T) =
\inf_{t \in [T,T+1)} \sum_{i} \int_{M} |\partial_t u_i|^2 (t) \leq  \lambda^{-1} c_0^{-1} C_0.
\end{split}
\]
Therefore, using $f_i = \sum_{j} a_{ij} \left( \frac{h_j e^{u_j}}{\int h_j e^{u_j}} - 1\right)$, and by \eqref{eq nabla u bdd}, we have
\[
\begin{split}
    &\sum_{i} \|\Delta u_i(t_T)\|_{L^2(M)}^2 \leq \sum_{i} \|\partial_t u_i(t_T)\|_{L^2(M)}^2 + \sum_{i} \|f_i(t_T)\|_{L^2(M)}^2\\
    \leq& \lambda^{-1} c_0^{-1} C_0 + nC |A|^2 \sum_{j} \left( \frac{M_j^2}{\|h_j^{1/q}\|_{L^{1}}^{2q}} \left( C_{TM} e^{\frac{1}{(q-1)^2 8\pi}  C_0}\right)^{2(q-1)}  C_{TM} e^{\frac{1}{2\pi}  C_0} + 1\right)
\end{split}
\]
as in the proof of Lemma \ref{lem ball}.
Here, we replace $C_2 R$ by $ C_0$.


Now by Sobolev embedding $W^{2,2} \hookrightarrow L^{\infty}$ and elliptic regularity, (or Calderon-Zygmund theory), we obtain
\[
\begin{split}
&\sum_{i} \|u_i(t_T)\|_{L^{\infty}(M)}^2\\
\leq& C \sum_{i} \|u_i(t_T)\|_{W^{2,2}(M)}^2\\
    \leq& C \sum_i \left( \|\Delta u_i(t_T)\|_{L^2(M)}^2 + \|u_i(t_T)\|_{L^2(M)}^2\right)\\
    \leq&  C \left( \lambda^{-1} c_0^{-1} C_0 + nC |A|^2 \sum_{j} \left( \frac{M_j^2}{\|h_j^{1/q}\|_{L^{1}}^{2q}} \left( C_{TM} e^{\frac{1}{(q-1)^2 8\pi}  C_0}\right)^{2(q-1)}  C_{TM} e^{\frac{1}{2\pi}  C_0} + 1\right) \right) + C  C_0\\
    =:& C_4^2.
\end{split}
\]
This completes the proof.
\end{proof}

Now we are ready to prove uniform boundedness for $u_i$ using Moser iteration.

\begin{thm} \label{thm unif bound}
    Let $u_i$ be a solution of \eqref{eq IVP2} on $M_T$.
    Then there exists a constant $C_5>0$ independent on $T$ such that for all $t \in (0,T)$,
    \begin{equation}
        \|u_i(t)\|_{L^{\infty}(M)} \leq C_5.
    \end{equation}
\end{thm}

\begin{proof}
As in the beginning of this section, we have
\[
\begin{split}
\partial_t u_i 
\leq& \Delta u_i + |A| \sum_j \frac{M_j}{\|h_j^{1/q}\|_{L^{1}}^{q}} \left( C_{TM} e^{\frac{1}{(q-1)^2 8\pi} C_0}\right)^{q-1}   \sum_j \left( e^{u_j} + \int e^{u_j}\right).
\end{split}
\]
Fix $\gamma \in (\frac{1}{2},1)$.
As above, for $p \geq 1$, multiply with $u_i^{2p-1}$, sum with $i$, and integrate to get
\[
\begin{split}
    &\frac{d}{dt} \left( \frac{1}{2p}  \sum_{i} \int_{M} u_i^{2p} \right)\\
    \leq& \sum_i \int_{M} \Delta u_i u_i^{2p-1} \\
    &+ |A| \sum_j  \frac{M_j}{\|h_j^{1/q}\|_{L^{1}}^{q}} \left( C_{TM} e^{\frac{1}{(q-1)^2 8\pi}  C_0}\right)^{q-1}  \sum_{i,j} \left( \int_{M} e^{u_j} u_i^{2p-1} + \int_{M} e^{u_j} \int_{M} u_i^{2p-1}\right).
\end{split}
\]
Now using $\int_{M} \Delta u_i u_i^{2p-1} = -\frac{1}{p} \int_{M} |\nabla (u_i^p)|^2$ and H{\"o}lder inequalities
\[
\begin{split}
\sum_{i,j}\int_M a_i b_j \leq& \left( \sum_{i,j} \int_{M} a_i^p \right)^{\frac{1}{p}} \left( \sum_{i,j} \int_{M} b_j^q \right)^{\frac{1}{q}} \leq n \left( \sum_{i} \int_{M} a_i^p \right)^{\frac{1}{p}}  \left( \sum_{j} \int_{M} b_j^q \right)^{\frac{1}{q}}\\
\sum_{i,j} \int_{M} a_i \int_{M} b_j \leq& \left( \sum_{i} \int_{M} a_i^p \right)^{\frac{1}{p}}  \left( \sum_{j} \int_{M} b_j^q \right)^{\frac{1}{q}},
\end{split}
\]
we get, by \eqref{eq e^pu},
\[
\begin{split}
    &\frac{d}{dt} \left( \frac{1}{2}\sum_i \int_{M} u_i^{2p} \right) +  \sum_i \int_{M} |\nabla (u_i^p)|^2\\
    \leq& (n+1)p |A| \sum_j \frac{M_j}{\|h_j^{1/q}\|_{L^{1}}^{q}} \left( C_{TM} e^{\frac{1}{(q-1)^2 8\pi} C_0}\right)^{q-1}  \left( \sum_j \int_{M} e^{\frac{2\gamma}{2\gamma-1}u_j} \right)^{\frac{2\gamma-1}{2\gamma}}  \left( \sum_i \int_{M} u_i^{4\gamma (p-\frac{1}{2})}\right)^{\frac{1}{2\gamma}}\\
    =:& p C_{\gamma} \left( \sum_i \int_{M} u_i^{4\gamma (p-\frac{1}{2})}\right)^{\frac{1}{2\gamma}}.
\end{split}
\]
Integrating over $[0,t] \subset [0,T]$ and taking supremum over $t$, we have
\begin{equation} \label{eq 2p_1}
\begin{split}
\sup_{t \in [0,T]} \sum_i \frac{1}{2} \int_{M} u_i^{2p}(t) &+ \sum_i \int_{0}^{T} \int_{M} |\nabla (u_i^p)|^2 \\
\leq& \sum_i \frac{1}{2} \int_{M} u_i^{2p}(0) + p C_\gamma T^{\frac{2\gamma-1}{2\gamma}} \left( \sum_i \int_{0}^{T} \int_{M} u_i^{4\gamma (p-\frac{1}{2})}\right)^{\frac{1}{2\gamma}}.
\end{split}
\end{equation}

Let $v_i = u_i^p$.
Choose a finite covering $ \{B_{2r}^{\ell} \}_{\ell=1}^{L}$ of $M$ consisting of balls of radius $2r$ such that $ \{B_{r}^{\ell} \}_{\ell=1}^{L}$ is also a covering of $M$.
Fix one of the balls $B_{2r}$ and a cut-off function $\phi \in C_c^{\infty}(B_{2r})$ such that $\phi \equiv 1$ on $B_{r}$ and that $|\nabla \phi| \leq \frac{2}{r}$.
From the Gagliardo-Nirenberg inequality $\|f\|_{L^2} \leq C_S\|\nabla f\|_{L^1}$ for all compactly supported $f \in W^{1,1}_0$, we have
\[
\begin{split}
    \left( \int_{B_{2r}} v_i^4 \phi^4 \right)^{\frac{1}{2}} \leq& C_S \int_{B_{2r}} |\nabla (v_i^2 \phi^2)| \leq 2C_S \int_{B_{2r}} v_i |\nabla v_i| \phi^2 + v_i^2 \phi |\nabla \phi|\\
    \leq& 2C_S \left( \left( \int_{B_{2r}} v_i^2 \phi^2 \right)^{\frac{1}{2}} \left( \int_{B_{2r}} |\nabla v_i|^2 \phi^2\right)^{\frac{1}{2}} + \frac{2}{r} \int_{B_{2r}} v_i^2 \right)\\
    \sum_{i}\int_{B_{r}} v_i^4 \leq& 8C_S^2 \left( \left(  \sup_{t \in [0,T]} \sum_i\int_{M} v_i^2 \right) \sum_i \int_{M} |\nabla v_i|^2 + \frac{4}{r^2} \left(  \sup_{t \in [0,T]}\sum_i \int_{M} v_i^2 \right)^2 \right) 
\end{split}
\]
Adding over all balls and integrating over $[0,T]$, we finally get
\[
\begin{split}
    \sum_i \iint_{M_T} v_i^4 \leq& 8LC_S^2  \left( \sup_{t \in [0,T]} \sum_i \int_{M} v_i^2 \right) \sum_i\iint_{M_T} |\nabla v_i|^2 \\
    & + 8LC_S^2  \frac{4T}{r^2} \left( \sup_{t \in [0,T]} \sum_i\int_{M} v_i^2 \right)^2 \\
    \left( \sum_{i}\iint_{M_T} v_i^{4} \right)^{\frac{1}{2}} \leq&  \left( 8LC_S^2 \right)^{\frac{1}{2}}  \left( \sup_{t \in [0,T]} \sum_i \int_{M} v_i^2 \right)^{\frac{1}{2}} \left( \sum_i\iint_{M_T} |\nabla v_i|^2 \right)^{\frac{1}{2}} \\
    &+ (8LC_S^2)^{\frac{1}{2}} \left( \frac{4T}{r^2} \right)^{\frac{1}{2}} \left( \sup_{t \in [0,T]} \sum_i \int_{M} v_i^2\right)\\
    \leq& (8LC_S^2)^{\frac{1}{2}} \left(  \left( 1 + \left(\frac{4T}{r^2} \right)^{\frac{1}{2}} \right)\left(  \sup_{t \in [0,T]} \sum_i \int_{M} v_i^2 \right) + \sum_i \iint_{M_T} |\nabla v_i|^2 \right).
\end{split}
\]


Together with \eqref{eq 2p_1}, we have
\[
\begin{split}
\left( \sum_i \iint_{M_T} u_i^{4p} \right)^{\frac{1}{4p}} \leq &(8LC_S^2)^{\frac{1}{4p}} \left( 1 + \sqrt{4Tr^{-2}}  \right)^{\frac{1}{2p}} \\
& \cdot \left( \sum_i \frac{1}{2} \int_{M} u_{i,0}^{2p} + p C_\gamma T^{\frac{2\gamma -1}{2\gamma}} \left( \sum_i \iint_{M_T} u_i^{4\gamma (p-\frac{1}{2})}\right)^{\frac{1}{2\gamma}} \right)^{\frac{1}{2p}}.
\end{split}
\]

Now denote
\begin{equation}
U_p = \left( \sum_i \iint_{M_T} u_i^p\right)^{\frac{1}{p}}.
\end{equation}
Also denote $U_0 = \sup_p \sum_i \frac{1}{2} \|u_{i,0}\|_{L^{2p}} < \infty$.
Then the above inequality becomes
\begin{equation} \label{eq iter}
U_{4p} \leq (8LC_S^2)^{\frac{1}{4p}}  \left( 1 + \sqrt{4Tr^{-2}}  \right)^{\frac{1}{2p}} \left( U_0^{2p} + p C_\gamma T^{\frac{2\gamma-1}{2\gamma}} U_{4\gamma (p-\frac{1}{2})}^{2(p-\frac{1}{2})} \right)^{\frac{1}{2p}}.
\end{equation}

Now we have two cases.

{\underline{Case 1}}:

If there are $p_n \to \infty$ such that $U_0^{2p_n} \geq p_n C_\gamma T^{\frac{2\gamma-1}{2\gamma}} U_{4\gamma (p_n-\frac{1}{2})}^{2(p_n-\frac{1}{2})}$, then we have 
\[
U_{4p_n} \leq  (32LC_S^2)^{\frac{1}{4p_n}}  \left( 1 + \sqrt{4Tr^{-2}}  \right)^{\frac{1}{2p_n}} U_0 \leq  (32LC_S^2)^{\frac{1}{4 p_1}}  \left( 1 + \sqrt{4Tr^{-2}}  \right)^{\frac{1}{2p_1}} U_0 < \infty.
\]
This implies that, $\|u_i\|_{L^{\infty}(M_T)} \leq  (32LC_S^2)^{\frac{1}{4 p_1}}  \left( 1 + \sqrt{4Tr^{-2}}  \right)^{\frac{1}{2p_1}} U_0  < \infty$ by taking $n \to \infty$, and we are done.

{\underline{Case 2}}:

For the second case, we may assume that there exists $p_*>0$ such that for all $p \geq p_*$, $U_0^{2p} \leq p C_\gamma T^{\frac{2\gamma-1}{2\gamma}}U_{4\gamma (p-\frac{1}{2})}^{2(p-\frac{1}{2})}$.
Then \eqref{eq iter} becomes
\begin{equation} \label{eq iter2}
U_{4p} \leq C_{T,r,\gamma}^{\frac{1}{2p}} (2p)^{\frac{1}{2p}} U_{4\gamma(p-\frac{1}{2})}^{\frac{p-\frac{1}{2}}{p}}
\end{equation}
where
\[
C_{T,r,\gamma} = \sqrt{8L}C_S (1 + \sqrt{4Tr^{-2}}) C_\gamma T^{\frac{2\gamma-1}{2\gamma}}.
\]
Fix $p_0 \geq p_*$.
Inductively we define $p_0 > p_1 > \ldots > p_{k+1} > 0$ by $p_{i+1} = \gamma (p_i - \frac{1}{2})$ and $0 < 4p_{k+1} \leq 2$.
For any $i=1, \ldots, k+1$, we can rewrite
\[
p_i = \gamma^i p_0 - \frac{\gamma^i + \gamma^{i-1} + \ldots + \gamma}{2} = \gamma^i p_0 - \frac{\gamma}{2} \frac{1-\gamma^i}{1-\gamma}.
\]
And at $i=k+1$, we have that
\begin{equation} \label{eq gamma k}
0 < 4p_{k+1} \leq 2 \Longleftrightarrow 0 < 2\gamma^k p_0 - \frac{1 - \gamma^{k+1}}{1-\gamma} \leq \frac{1}{\gamma} \Longleftrightarrow \gamma < \gamma^{k+1} (\gamma + 2(1-\gamma)p_0) \leq 1.
\end{equation}

Then \eqref{eq iter2} becomes
\[
\begin{split}
U_{4p_0} \leq& C_{T,r,\gamma}^{\frac{1}{2p_0}} (2p_0)^{\frac{1}{2p_0}} U_{4p_1}^{\frac{p_1}{\gamma p_0}}\\
\leq& C_{T,r,\gamma}^{\frac{1}{2p_0}} (2p_0)^{\frac{1}{2p_0}} \left( C_{T,r,\gamma}^{\frac{1}{2p_1}} (2p_1)^{\frac{1}{2p_1}}U_{4p_2}^{\frac{p_2 }{\gamma p_1}} \right)^{\frac{p_1}{\gamma p_0}}\\
 =& C_{T,r,\gamma}^{\frac{1}{2p_0} + \frac{1}{2\gamma p_0}} (2p_0)^{\frac{1}{2p_0}} (2p_1)^{\frac{1}{2 \gamma p_0}} U_{4p_2}^{\frac{p_2}{\gamma^2 p_0}}.
\end{split}
\]
By induction, we get
\[
U_{4p_0} \leq C_{T,r,\gamma}^{\frac{1}{2p_0} + \frac{1}{2\gamma p_0} + \ldots + \frac{1}{2 \gamma^k p_0}}  (2p_0)^{\frac{1}{2p_0}} (2p_1)^{\frac{1}{2 \gamma p_0}} \cdot \ldots \cdot (2p_k)^{\frac{1}{2 \gamma^k p_0}} U_{4p_{k+1}}^{\frac{p_{k+1}}{\gamma^{k+1} p_0}}.
\]

For the first exponent, by \eqref{eq gamma k}, note that
\[
\begin{split}
    \frac{1}{2p_0} + \frac{1}{2\gamma p_0} + \ldots + \frac{1}{2 \gamma^k p_0} =& \frac{1}{2p_0} \left( 1 + \gamma^{-1} + \ldots + \gamma^{-k} \right)\\
    =& \frac{1}{2p_0} \frac{\gamma^{-(k+1)} - 1}{\gamma^{-1}-1}\\
    <& \frac{1}{2p_0} \frac{1}{\gamma^{-1} - 1} \frac{ 2(1-\gamma) p_0}{\gamma} = 1.
\end{split}
\]
For the last exponent, again by \eqref{eq gamma k}, we get
\[
\begin{split}
    \frac{p_{k+1}}{\gamma^{k+1} p_0} =& \frac{\gamma^{k+1} p_0 - \frac{\gamma}{2} \frac{1-\gamma^{k+1}}{1-\gamma}}{\gamma^{k+1} p_0} < 1.
\end{split}
\]
From \eqref{eq poincare} and $4p_{k+1} \leq 2$, we have $U_{4p_{k+1}}^{\frac{p_{k+1}}{\gamma^{k+1} p_0}} \leq C(T)$ which is independent on $p_0$.

To see the middle factors, from \eqref{eq gamma k}, note that
\[
\begin{split}
 0 < \frac{2p_{k+1}}{\gamma} =& 2\gamma^{k} p_0 - (\gamma^{k} + \gamma^{k-1} + \ldots + 1)\leq \gamma^{-1}\\
\Longleftrightarrow \gamma^{-1} < \frac{2p_k}{\gamma} =& 2 \gamma^{k-1} p_0 - (\gamma^{k-1} + \ldots + 1) \leq \gamma^{-1} + \gamma^{-2}\\
\Longleftrightarrow \gamma^{-1} + \gamma^{-2} < \frac{2p_{k-1}}{\gamma} =& 2 \gamma^{k-2} p_0 - (\gamma^{k-2} + \ldots + 1) \leq \gamma^{-1} + \gamma^{-2} + \gamma^{-3}.
\end{split}
\]
For simplicity, denote $\mu = \gamma^{-1}$.
Then for any $i=0,1, \ldots, k$,
\[
1 + \mu + \ldots + \mu^{k-i} < 2p_i \leq 1 + \mu + \ldots + \mu^{k-i} + \mu^{k+1-i}.
\]
Then
\[
\begin{split}
     \prod_{i=0}^{k} (2p_i)^{\frac{1}{2 \gamma^i p_0}} \leq&  \prod_{i=0}^{k} \left( 1 + \ldots + \mu^{k+1-i} \right)^{\frac{1}{2 \gamma^{i} p_0}} < \prod_{i=0}^{k} \left(\frac{1}{\mu-1} \mu^{k+2-i} \right)^{\frac{1}{2 \gamma^{i} p_0}}\\
     <& \left(\frac{1}{\mu-1} \right) \left(\prod_{i=0}^{k} \mu^{\frac{k+2-i}{\gamma^i}} \right)^{\frac{1}{2p_0}}.
\end{split}
\]
The product becomes sum in the exponent, which is
\[
\begin{split}
\sum_{i=0}^{k} \frac{k+2-i}{\gamma^i} =& \left(2 \mu^k + 3 \mu^{k-1} + \ldots + (k+2) \cdot1 \right)\\
=& 2\left( \mu^k + \mu^{k-1} + \ldots + 1 \right) + \left( \mu^{k-1} + \mu^{k-2} + \ldots + 1\right)\\
& + \ldots + (\mu + 1) + 1\\
<& 2 \frac{\mu}{\mu-1} \mu^{k} + \frac{\mu}{\mu-1} \mu^{k-1} + \ldots + \frac{\mu}{\mu-1} \mu + \frac{\mu}{\mu-1}\\
<& \frac{\mu}{\mu-1} \left( \mu^{k} + \frac{\mu}{\mu-1} \mu^{k} \right) = \frac{\mu(2 \mu-1)}{(\mu-1)^2} \mu^{k}.
\end{split}
\]
Hence, using \eqref{eq gamma k}, we have
\[
\begin{split}
    \prod_{i=0}^{k} (2p_i)^{\frac{1}{2 \gamma^i p_0}} \leq&  \left( \frac{1}{\mu-1}\right) \left( \mu^{\frac{\mu(2\mu-1)}{(\mu-1)^2} \mu^{k}} \right)^{\frac{1}{2p_0}} =\left( \frac{1}{\mu-1}\right) \left( \mu^{\frac{\mu(2\mu-1)}{(\mu-1)^2}} \right)^{\frac{1}{2 \gamma^k p_0}}\\
    \leq& \left( \frac{1}{\mu-1}\right) \left( \mu^{\frac{\mu(2\mu-1)}{(\mu-1)^2}} \right)^{\frac{\gamma + 2(1-\gamma) p_0}{2 p_0}} \leq \left( \frac{1}{\mu-1}\right) \left( \mu^{\frac{\mu(2\mu-1)}{(\mu-1)^2}} \right)^{\frac{\gamma + 2(1-\gamma) p_*}{2 p_*}} < \infty.
\end{split}
\]

In conclusion, we get
\[
U_{4p_0} \leq C(T)
\]
where RHS is independent on $p_0$.
By taking $p_0 \to \infty$, we get
\[
\|u_i\|_{L^{\infty}(M_T)} \leq C(T).
\]

Combining two cases, we conclude that
\[
\begin{split}
\|u_i\|_{L^{\infty}(M_T)} \leq& (32LC_S^2)^{\frac{1}{4 p_1}} \left(1 + \sqrt{4Tr^{-2}} \right)^{\frac{1}{2p_1}} U_0  + C(T)\\
\leq& (32LC_S^2)^{\frac{1}{4p_1}} \left(1 + \sqrt{4Tr^{-2}} \right)^{\frac{1}{2p_1}}C \|u_{i,0}\|_{L^{\infty}(M)} + C(T).
\end{split}
\]



Finally, we will get a uniform estimate independent on $T$.
We first consider the case that $T<2$.
Then we have
\[
\begin{split}
    \sup_{s \in [0,T]} \|u_i(s)\|_{L^{\infty}(M)} \leq& (32LC_S^2)^{\frac{1}{4p_1}} \left(1 + \sqrt{4(2)r^{-2}} \right)^{\frac{1}{2p_1}}C \|u_{i,0}\|_{L^{\infty}(M)} + C (2)\\
    =:& C_5 < \infty.
\end{split}
\]
Next consider the case that $T \geq 2$.
By Lemma \ref{lem unif bound-T}, we have that for any $T' \in (0,T-2)$,
\[
\begin{split}
    \sup_{s \in [T'+1,T'+2]} \|u_i(s)\|_{L^{\infty}(M)} \leq& \sup_{t \in [t_{T'},t_{T'}+2)} \|u_i(t)\|_{L^{\infty}(M)}\\
    \leq& (32LC_S^2)^{\frac{1}{4p_1}} \left(1 + \sqrt{4(2)r^{-2}} \right)^{\frac{1}{2p_1}}C \|u_i(t_{T'})\|_{L^{\infty}(M)} + C (2)\\
    =:& C_5 < \infty.
\end{split}
\]
Remaining estimate $\sup_{s \in [0,1]}  \|u_i(s)\|_{L^{\infty}(M)}$ can be estimated as the first case.
This completes the proof.
\end{proof}

\begin{lem} \label{lem uniq}
(Uniqueness)
Let $u_{i,0} \in W^{2,p}(M)$ and $p>2$.
Then the solution $u_i$ in \ref{thm short time} is unique.
\end{lem}

\begin{proof}
Suppose $u_i, \tilde{u}_i$ are solutions in \ref{thm short time}.
Then $v_i := u_i - \tilde{u}_i$ satisfies
\begin{equation}
    \partial_t v_i = \Delta v_i + (f_i - \tilde{f}_i)
\end{equation}
where
    \[
    f_i = \sum_{j} a_{ij} \left(\frac{h_j e^{u_j}}{\int h_j e^{u_j}} - 1 \right), \quad \tilde{f}_i = \sum_{j} a_{ij} \left(\frac{h_j e^{\tilde{u}_j}}{\int h_j e^{\tilde{u}_j}} - 1 \right)
    \]
and the initial condition $v_i(0) = u_i(0) - \tilde{u}_i(0) = 0$.
As in the proof of Theorem \ref{thm short time}, and using \eqref{eq h e^v lower2} replacing $R$ by $C_5$, we get
\[
\begin{split}
\int_{M}|f_i - \tilde{f}_i|^2 \leq& 2C |A|^2 \sum_{j} \frac{ M_j^{2} e^{2C_5}}{\|h_j\|_{L^{1/q}}^{2}} \int_{M} e^{2C_5} |\tilde{u_j}-u_j|^2 \int_{M} e^{2C_5}\\
\leq& 2C |A|^2 e^{4C_5} \sum_{j} \frac{ M_j^{2} e^{2C_5}}{\|h_j\|_{L^{1/q}}^{2}} \|v_j\|_{L^2(M)}^2\\
\leq& 2C |A|^2 e^{4C_5} \sum_{j} \frac{ M_j^{2} e^{2C_5}}{\|h_j\|_{L^{1/q}}^{2}} \sum_{k}\|v_k(t)\|_{L^2(M)}^2\\
\end{split}
\]
where we use Theorem \ref{thm unif bound}.
Then we have
\[
\begin{split}
    \frac{d}{dt} \frac{1}{2}\sum_{i}\|v_i(t)\|_{L^2(M)}^2 =& \sum_{i} \int_{M} v_i(t) \partial_t v_i(t)\\
    =& \sum_{i} \int_{M} v_i(t) (\Delta v_i(t) + (f_i - \tilde{f}_i))\\
    \leq& -\sum_{i} \int_{M} |\nabla v_i|^2(t) + \frac{1}{2} \sum_{i} \|v_i(t)\|_{L^2(M)}^2 + \frac{1}{2} \sum_{i}\int_{M} |f_i - \tilde{f}_i|^2\\
    \leq&  \left( \frac{1}{2} +  nC |A|^2 e^{4C_5} \sum_{j} \frac{ M_j^{2} e^{2C_5}}{\|h_j\|_{L^{1/q}}^{2}}    \right) \sum_{i} \|v_i(t)\|_{L^2(M)}^2.
\end{split}
\]
Now the function $X(t) = \sum_{i} \|v_i(t)\|_{L^2(M)}^2$ satisfies
\[
X'(t) \leq \beta X(t)
\]
for some constant $\beta>0$ with the initial condition $X(0)=0$.
Hence, by Gronwall's inequality, we get $X(t) \equiv 0$ which implies $u_i \equiv \tilde{u}_i$.

\end{proof}

\begin{thm} \label{thm global exist}
(Global existence)
The initial value problem \eqref{eq IVP2} admits a unique global solution $u \in C([0,\infty),W^{1,2}(M)) \cap C^{\infty}(M \times (0,\infty))$.
\end{thm}
The notation $u \in C([0,\infty),W^{1,2}(M)) \cap C^{\infty}(M \times (0,\infty))$ means for each $t\in [0,\infty)$, $u(t)\in W^{1,2}(M)$ and $\|u(t)\|_{W^{1,2}(M)}$ is continuous in $t$ and is $C^{\infty}$ in $(0,\infty)\times M$.
\begin{proof}
    It follows from a standard continuation argument using Theorem \ref{thm short time}, Lemma \ref{lem unif bound-T}, Lemma \ref{lem uniq}, and with explicit control of $T$ in \eqref{eq cond T}.
\end{proof}
Theorem \ref{main-thm-1} and Corollary \ref{main-thm-2} follow immediately.


\begin{thebibliography}{99}
\bibitem{biler} Biler, Piotr; Nadzieja, Tadeusz;
Existence and nonexistence of solutions for a model of gravitational interaction of particles. I.
Colloq. Math. 66 (1994), no. 2, 319–334.
\bibitem{borer} Borer, Franziska; Elbau, Peter; Weth, Tobias
A variant prescribed curvature flow on closed surfaces with negative Euler characteristic.
Calc. Var. Partial Differential Equations 62 (2023), no. 9, Paper No. 262, 34 pp.

\bibitem{chang-ding-ye} Chang, Kung-Ching; Ding, Wei Yue; Ye, Rugang,
Finite-time blow-up of the heat flow of harmonic maps from surfaces.
J. Differential Geom. 36 (1992), no. 2, 507–515.


\bibitem{childress} S. Childress and J. K. Percus, Nonlinear aspects of Chemotaxis,
Math. Biosci. 56 (1981), 217--237.

\bibitem{CSW} M. Chipot, I. Shafrir, G. Wolansky, On the solutions of Liouville systems.  J. Differential Equations  140  (1997),  no. 1, 59--105.
\bibitem{CSW1} M. Chipot, I. Shafrir, G. Wolansky, Erratum: ``On the solutions of Liouville systems'' [J. Differential Equations 140 (1997), no. 1, 59--105; MR1473855 (98j:35053)].  J. Differential Equations  178  (2002),  no. 2, 630.


\bibitem{gu-zhang} Gu, Yi; Zhang, Lei Degree counting theorems for singular Liouville systems. Ann. Sc. Norm. Super. Pisa Cl. Sci. (5) 21 (2020), 1103–1135.

\bibitem{gu-zhang-2} Gu,Yi; Zhang, Lei, Structure of bubbling solutions of Liouville systems with negative singular sources. To appear on Communications in Contemporary Mathematics.  arxiv.org. 2112.10031

\bibitem{huang-cvpde} H. Huang Existence of bubbling solutions for the Liouville system in a torus. Calc. Var. Partial Differential Equations 58 (2019), no. 3, Paper No. 99, 26 pp.


\bibitem{hong} J. Hong, Y. Kim and P. Y. Pac, Multivortex solutions of the Abelian Chern-Simons-Higgs theory,
Phys. Rev. Letter 64 (1990), 2230--2233.

\bibitem{jackiw} R. Jackiw and E. J. Weinberg, Selfdual Chern Simons vortices,
Phys. Rev. Lett. 64 (1990), 2234--2237.


\bibitem{keller} E. F. Keller and L. A. Segel, Traveling bands of Chemotactic Bacteria: A theoretical analysis,
J. Theor. Biol. 30 (1971), 235--248.

\bibitem{kiess} M. K.-H. Kiessling, Statistical mechanics of classical particles with logarithmic interactions,
Comm. Pure Appl. Math. 46 (1993), no. 1, 27--56.


\bibitem{kimleelee}C. Kim, C. Lee and B.-H. Lee, Schr\"odinger fields on the plane with $[U(1)]^N$ Chern-Simons interactions and generalized self-dual solitons,
Phys. Rev. D (3) 48 (1993), 1821--1840.


\bibitem{lin-wei-ye} C. S. Lin, J. C. Wei, D. Ye, Classification and nondegeneracy of $SU(n+1)$  Toda system with singular sources.
Invent. Math. 190 (2012), no. 1, 169–207.
    
\bibitem{lin-zhang-1} C. S. Lin, L. Zhang, Profile of bubbling solutions to a Liouville system.  Ann. Inst. H. Poincare Anal. Non Lineaire 27 (2010), no. 1, 117--143.

\bibitem{lin-zhang-2} C. S. Lin, L. Zhang,   A topological degree counting for some Liouville systems of mean field equations, Comm. Pure Appl. Math. volume 64, Issue 4, pages 556--590, April 2011.

\bibitem{struwe-1} M. Struwe, On the evolution of harmonic mappings of Riemannian surfaces. 
Comment. Math. Helv. 60 (1985), no. 4, 558–581.

\bibitem{struwe-2}M. Struwe, ``Bubbling'' of the prescribed curvature flow on the torus.
J. Eur. Math. Soc. (JEMS) 22 (2020), no. 10, 3223–3262.

\bibitem{takasaki}Takasaki, Kanehisa, 
Toda hierarchies and their applications.
J. Phys. A 51 (2018), no. 20, 203001, 35 pp.



\bibitem{Wil} F. Wilczek,   Disassembling anyons.  Physical review letters 69.1 (1992): 132.

\bibitem{wolansky1} G. Wolansky, On steady distributions of self-attracting clusters under friction and fluctuations,
Arch. Rational Mech. Anal. 119 (1992), 355--391.

\bibitem{wolansky2} G. Wolansky, On the evolution of self-interacting clusters and applications to semi-linear equations with exponential nonlinearity,
J. Anal. Math. 59 (1992), 251--272.


\bibitem{wolansky3} G. Wolansky. Multi-components chemotactic system in the absence of conflicts. European Journal of Applied Mathematics, Volume 13, Issue 6, 2002.

\bibitem{wu-zhuoqun}Wu, Zhuoqun; Yin, Jingxue; Wang, Chunpeng,  Elliptic\& parabolic equations. World Scientific Publishing Co. Pte. Ltd., Hackensack, NJ, 2006, xvi+408 pp.

\bibitem{yang} Y. Yang, Solitons in field theory and nonlinear analysis,
Springer-Verlag, 2001.


\end{thebibliography}
\end{document}